\documentclass[a4paper,12pt]{article}
\usepackage[left=2.5cm,right=2.5cm,bottom=2.5cm,top=2.5cm]{geometry}

\usepackage{amssymb,amsmath,amsthm,mathtools}%,mathptmx}

\usepackage[utf8]{inputenc}
\usepackage[english]{babel}
\usepackage[T1]{fontenc}

\usepackage{bm}
\usepackage{color}
\usepackage{hyperref}
\usepackage{comment}

\usepackage{mathrsfs}

\newcommand{\de}{\partial}
\newcommand{\vc}[1]{\boldsymbol{#1}}
\newcommand{\vt}[1]{\mathsf{#1}}
\newcommand{\tsp}{\mathsf{T}}

\newcommand{\tr}{\operatorname{\mathrm{tr}}}
\newcommand{\dvg}{\operatorname{\mathrm{div}}}

\newcommand{\mat}{\mathrm{Mat}}

\newcommand{\Bref}{{\vt{B}_\mathrm{ref}}}
\newcommand{\Brefn}{\B_{\mathrm{ref},n}}
\newcommand{\BB}{{\mathcal{B}}}
\newcommand{\F}{{\vt{F}}}
\newcommand{\G}{{\vt{G}}}

\newcommand{\M}{{\vt{M}}}
\newcommand{\MM}{{\mathcal{M}}}

\newcommand{\W}{{\vt{W}}}

\newcommand{\I}{{\vt{I}}}

\newcommand{\adv}[1]{\mathcal{D}_{\vc #1}}

\newcommand{\pl}{{\boldsymbol{\varphi}}}

\newcommand{\dint}{\int_0^{T}\mskip-8mu\int_\Omega}

\newcommand{\weakto}{\rightharpoonup}
\newcommand{\weaktos}{\stackrel{*}{\rightharpoonup}}
\newcommand{\dee}{\mathrm{d}}
\newcommand{\vv}{\vc{v}}
\newcommand{\uu}{\vc{u}}
\newcommand{\taur}{\tau_\mathrm{r}}

\newcommand{\tT}{{\vt{T}}}
\newcommand{\B}{{\vt{B}}}
\newcommand{\tC}{{\vt{C}}}
\newcommand{\bR}{\mathbb R}

\newtheorem{thm}{Theorem}[section]
\newtheorem{prop}[thm]{Proposition}

\newtheorem{lem}[thm]{Lemma}

\theoremstyle{remark}
\newtheorem{rem}[thm]{Remark}

\theoremstyle{definition}

\newtheorem{definition}[thm]{Definition}

\usepackage{authblk}

\begin{document}

\title{Viscoelasticity, logarithmic stresses, and tensorial transport equations}

\author{Gennaro Ciampa$^1$\footnote{gennaro.ciampa@unimi.it}, Giulio G.~Giusteri$^2$\footnote{giulio.giusteri@unipd.it}, Alessio G.~Soggiu$^2$}
\affil{$^1$DISIM - Dipartimento di Ingegneria e Scienze dell'Informazione e Matematica, Universit\`a degli Studi dell'Aquila, Via Vetoio, 67100, L'Aquila, Italy}
\affil{$^2$Dipartimento di Matematica ``Tullio Levi-Civita'', Universit\`a degli Studi di Padova, Via Trieste 63, 35131, Padova, Italy}
\date{\today}

\begin{comment}

%ORCID: 0000-0001-9001-9706

%%%For Springer Nature
\author[1]{\fnm{Gennaro} \sur{Ciampa}}\email{gennaro.ciampa@unimi.it}

\author*[2]{\fnm{Giulio G.} \sur{Giusteri}}\email{giulio.giusteri@unipd.it}

\author[2]{\fnm{Alessio G.} \sur{Soggiu}}\nomail%\email{alessio.soggiu@hotmail.it}


\affil[1]{\orgdiv{DISIM - Dipartimento di Ingegneria e Scienze dell'Informazione e Matematica}, \orgname{Universit\`a degli Studi dell'Aquila}, \orgaddress{\street{Via Vetoio}, \city{L'Aquila}, \postcode{67100}, \country{Italy}}}

\affil*[2]{\orgdiv{Dipartimento di Matematica ``Tullio Levi-Civita''}, \orgname{Universit\`a degli Studi di Padova}, \orgaddress{\street{Via Trieste 63}, \city{Padova}, \postcode{35121}, \country{Italy}}}

\end{comment}

\maketitle

%For Springer Nature
%\abstract{
\begin{abstract}%
We introduce models for viscoelastic materials, both solids and fluids, based on logarithmic stresses to capture the elastic contribution to the material response. The matrix logarithm allows to link the measures of strain, that naturally belong to a multiplicative group of linear transformations, to stresses, that are additive elements of a linear space of tensors. As regards the viscous stresses, we simply assume a Newtonian constitutive law, but the presence of elasticity and plastic relaxation makes the materials non-Newtonian. Our aim is to discuss the existence of weak solutions for the corresponding systems of partial differential equations in the nonlinear large-deformation regime. The main difficulties arise in the analysis of the transport equations necessary to describe the evolution of tensorial measures of strain. For the solid model, we only need to consider the equation for the left Cauchy--Green tensor, while for the fluid model we add an evolution equation for the elastically-relaxed strain. Due to the tensorial nature of the fields, available techniques cannot be applied to the analysis of such transport equations. To cope with this, we introduce the notion of charted weak solution, based on non-standard a priori estimates, that lead to a global-in-time existence of solutions for the viscoelastic models in the natural functional setting associated with the energy inequality.
\textbf{MSC Classification (2020):} 35Q35 - 35Q74 - 35D99 - 76A10 - 74D99
\end{abstract}%}

%\keywords{keyword1, Keyword2, Keyword3, Keyword4}

%\pacs[MSC Classification]{???, ???}

%\begin{abstract}
%\red{TBW (150 to 250 words)}
%\end{abstract}

%\end{document}
%%%%%%%%%%%%%%%%%%%%%%%%%%%%%%%
\section{Introduction}

The realm of viscoelasticity covers a broad class of phenomena in continuum mechanics relevant to the description of materials that, when deformed, develop internal stresses of both conservative and dissipative nature. The two most successful continuum theories to date, namely solid elasticity and Newtonian fluid mechanics, represent opposite conditions in which either dissipative or conservative effects are completely negligible~\cite{Gurtin_2010}. 
As a consequence, the mathematical techniques employed to treat the two classes of models developed quite separately. Variational methods became the tool of choice to deal with elasticity, as testified by numerous authoritative monographs~\cite{Ciarlet_1988, Ogden_2001, Antman_2005}, while the analysis of nonlinear partial differential equations marked the history of mathematical fluid mechanics~\cite{Ladyzhenskaya_1969,Lions_1969,Temam_1984, Lions_1996, Lions_1998}. In a similar fashion, the treatment of viscoelasticity has followed different paths when stemming from solid or fluid mechanics.

Common to all viscoelastic frameworks is the need to keep track of both the material deformation, related to elastic responses, and the rate of deformation, involved in dissipative effects. We work in a setting typical of the fluid mechanics approach but also covering models of viscoelastic solids, which will be indeed our starting point. As a matter of fact, viscoelastic fluid models can also describe visco-elasto-plastic solids in certain regimes.
We use the spatial, or Eulerian, framework and describe viscous stresses by the simplest Newtonian model that leads to the Navier--Stokes equation. The distinctive features of the models that we introduce are the use of stresses proportional to the logarithm of the strain that keeps track of the material deformation and of evolution equations for both the current and relaxed strain measures, as opposed to the single evolution equation for the elastic stress or the conformation tensor that characterizes Oldroyd-type models of viscoelastic fluids~\cite{Oldroyd_1950,Saut_2013,Renardy_2021}. The combination of these two aspects, which have a rather long history in solid mechanics and plasticity theory, has important consequences on the mathematical analysis of the evolution equations for our viscoelastic models.

The importance of considering the logarithm of strain measures has been recently highlighted by several works \cite{Xiao_2000, Xiao_2004,Neff_2016,Prusa_2020}. It is intimately related to acknowledging that the mechanical interpretation of those tensors identifies them as elements of submanifolds of Lie groups. The matrix logarithm provides a mapping to the associated Lie algebra that is key in the development of constitutive relations.
In fact, the Cauchy stress tensor, from which the deformation of the material is driven, can be naturally interpreted as an element of the tangent or cotangent bundle to the Banach manifold of tensorial measures of strain, directly related to Lie algebrae and logarithmic strains.
What may appear a subtle mathematical argument eventually leads to a more natural balance of the terms that couple the evolution equations for the material to those for the tensorial measures of strain, as we shall highlight in due course.

Another important aspect of this work is that we need to consider transport equations for tensor fields that are coupled to the balance of linear momentum, since the strains that affect the elastic stress are advected by the velocity field of the continuum. This is common to several models for viscoelastic materials featuring the evolution of tensorial measures of strain or stress~\cite{Bird_1987,Larson_1992,Saut_2013,Renardy_2021}, but we propose a somewhat different approach to the analysis of such equations. 

Transport equations are a classical topic in analysis and mathematical physics for their ubiquitous presence in continuum mechanical models where the evolution of various fields is coupled with the deformation of the material. 
This class of evolution equations is still attracting considerable attention in relation to the existence and regularity of solutions.
They have stimulated the introduction of novel concepts of weak solutions and, beyond the seminal paper by DiPerna \& Lions~\cite{DPL}, a number of works have been broadening the field in recent years (see for instance Refs.~\cite{Am,BB,CDL,MS3}).

Nevertheless, the tensorial transport equations we need to consider cannot be reduced to a system of scalar equations and a direct application of available results remains elusive. 
This is due to the presence of source terms (essential for the physical meaning of the model) that produce a nontrivial coupling between different components of the tensor and with the gradient of the transporting velocity field.
Moreover, and most importantly, the interpretation of those tensors within submanifolds of Lie groups poses serious limitations to the standard application of techniques based on Lebesgue and Sobolev spaces.

The existence of solutions in the nonlinear finite-deformation regime with natural regularity assumptions on the data is still an open problem for several viscoelastic models considered in the literature~\cite{Renardy_2021}, notwithstanding the important results of Masmoudi~\cite{Masmoudi_2011} on Oldroyd-type models.
Notably, Liu \& Walkington~\cite{LW} were able to establish an existence result by considering initial data that are small in the sense that the relevant tensor field is close to the identity.
This assumption, essentially, allows to linearize the elastic stress--strain relation and perturbatively split the transport equation for the deformation gradient in two decoupled equations.
The first one concerns a rotation field that evolves in a compact manifold, so that very strong \emph{a priori} estimates are available and permit a treatment \emph{\`a la} Di Perna--Lions of the source in the transport equation.
The second one follows an evolution on the tangent space to the manifold of deformation tensors and is thus amenable to a linear space analysis. 
Stronger convergence properties of approximate solutions were obtained, thanks to a crucial assumption, in a simplified model considered by Lions \& Masmoudi~\cite{LM}, who could then prove global-in-time existence of weak solutions. This inspired also the work of Bejaoui \& Majdoub~\cite{Bejaoui_2013} on other models.
Further important results on the viscoelastic dynamics in the vicinity of the elastically-relaxed state were obtained by Lin \& Zhang~\cite{Lin_2005}, Lin, Liu \& Zhang~\cite{Lin_2008}, and Lei, Liu \& Zhou~\cite{Lei_2008}.
Another approach to cope with the nonlinear coupling has been followed by Kalousek~\cite{Kalousek_2019} and consists in looking for solutions of the equations up to a reminder term, that may or may not be negligible depending on the size of initial data.
Also in this case, the relevance of the solution is proper in a somewhat linear or small-data regime.

We follow a rather different route and introduce the notion of \emph{charted weak solutions} that provides a way to tackle the analysis of tensorial transport equations with sources within a linear space setting. When applied to the study of evolution equations for viscoelastic materials, this approach generates solutions enjoying optimal regularity properties in relation to the natural estimates on initial data.
Moreover, we obtain global-in-time existence of solutions for arbitrarily large deformations.
We exploit the link between a Lie group and the associated Lie algebra through suitable logarithm and exponential maps and consider important cases in which the relevant tensor manifold, which is an infinite-dimensional Banach manifold, can be locally parametrized by a single chart in a Banach (or even Hilbert) space.
In this way, approximation and convergence properties on the manifold are reduced to those in the linear setting of the local chart. 
Charted weak solutions are based on an approximation scheme and compactness arguments. As of now, the precise meaning of the limit equation remains to be established. The situation resembles the one encounter in variational problems involving non-smooth functionals, where Euler--Lagrange equations may be ill-defined at minima.

In Section~\ref{sec:prototypical} we present a prototypical model for the dynamics of an incompressible viscoelastic material. This motivating example features tensorial transport equations the analysis of which requires some novel ideas.
We introduce the definition of charted weak solutions for our tensorial equation in Section~\ref{sec:equations_and_solutions}.
The proof of existence of suitably defined solutions for the dynamics of the prototypical model is given in Section~\ref{sec:existence_solid}, while Sections~\ref{sec:fluid_model} is devoted to the analysis of a more structured viscoelastic fluid model.
Prospective applications and further research directions are outlined in Section~\ref{sec:further_applications}, while Section~\ref{sec:open_problems} provides a discussion of some open problems.

%%%%%%%%%%%%%%%%%%%%%%%%%%%%%%%
\section{A prototypical model for viscoelastic materials}
\label{sec:prototypical}

Let us consider a set of labels represented by a bounded domain $\Omega_0 \subset \mathbb{R}^d$ (with  $d = 2$ or $3$) with Lipschitz boundary.
The points of the set of labels are identified by the Lagrangian (or material) coordinates $\vc{X} \in \Omega_0$. Given a time interval $[0,T]$ with $T>0$, we define a time-dependent deformation as the map
\[
\pl \colon \left\{
\begin{aligned}
 \mbox{}[0,T]\times \Omega_0 &\rightarrow \bR^d \\
(t,\vc{X}) \quad &\mapsto \pl(t,\vc{X})
\end{aligned}\right.
\]
and the deformation gradient tensor field 
$\hat{\F} \colon [0,T]\times\Omega_0 \rightarrow \mat_d(\bR)$
with components defined by
\[
\hat{\F} (t,\vc{X})_{ij} = \frac{\de \pl_i}{\de \vc{X}_j} (t,\vc{X}).
\]
The map $\pl$ represents the position at time $t$ of a material point labelled by $\vc X$.
For any $t\in[0,T]$ we assume that $\pl(t,\cdot)$ is injective and that $\det\hat{\F}(t,\vc X)>0$ for any $(t,\vc X)\in[0,T]\times\Omega_0$.
For now, we assume that the regularity of $\pl$ is such that $\hat{\F}$ is well defined.

In order to introduce the Eulerian setting without ambiguities, we limit ourselves to the case in which, for all times $t\in[0,T]$,  $\pl(t,\Omega_0)=\Omega\subset \bR^d$, where the fixed spatial domain $\Omega$ is bounded and with Lipschitz boundary. The points of $\Omega$ are the spatial coordinates $\vc x$.
Moreover, we define the spatial inverse on $\Omega$ of the deformation as
\[
\tilde{\pl} \colon \left\{
\begin{aligned}
	\mbox{}[0,T] \times \Omega &\rightarrow \Omega_0 \\
	(t,\vc{x})\quad &\mapsto \tilde{\pl}(t,\vc{x})
\end{aligned}
\right.
\]
with the property
$\tilde{\pl}(t,\pl(t,\vc X))=\vc X$.
The Eulerian velocity field $\vc{u}$ is defined by 
\[
\vc u(t,\vc x)\coloneqq{\de_t \pl}(t,\tilde{\pl}(t,\vc x)).
\]
In the Eulerian setting, we will write the evolution equations in terms of $\vc u$ and the Eulerian deformation gradient $\F$, given by $\F(t,\vc x)\coloneqq\hat{\F}(t,\tilde{\pl}(t,\vc x))$.
From the above definition we see that the map $\pl$ is the solution of the nonlinear ordinary differential equation $\dot{\pl}=\vc u(t,\pl)$.

\subsection{Incompressible evolution of the left Cauchy--Green tensor}\label{sec:incompressibility}

We define the advective derivative associated with a divergence-free velocity field $\vc u$ as
\[
\adv{u}\coloneqq{\de_t}+(\vc u\cdot\nabla).
\] 
The evolution equation for the Eulerian version $\F$ of the deformation gradient reads
\begin{equation}\label{eq:Fonly}
\adv{u}\F=\nabla\vc u\F.
\end{equation}
From this we can easily deduce the equation for the left Cauchy--Green tensor $\B\coloneqq\F\F^\tsp$ as
\begin{equation}\label{eq:Bonly}
\adv{u}\B=\nabla\vc u\B+\B\nabla\vc u^\tsp.
\end{equation}

For an incompressible material, we know that the velocity field is divergence-free and that the determinant of $\F$ is constant.
Without loss of generality, we can assume $\det\F(t,\vc x)=1$ for any instant in time and point in space, implying also $\det\B(t,\vc x)=1$ for any $(t,\vc x)\in[0,T]\times\Omega$. 
To ascertain that the evolution generated by equation \eqref{eq:Bonly} preserves the value of the determinant, it is enough to check that the right-hand side is tangent to the manifold of tensor fields with unit determinant.
To this end, we consider the first variation of the constraint functional
\[
\mathcal C[\M]\coloneqq\dint (\det\M-1)
\]
that gives, for any test tensor field $\G$,
\[
\langle\delta\mathcal C[\M],\G\rangle = \dint\tr\big(\mathrm{cof}(\M)^\tsp\G\big)=\dint\tr\big(\M^{-1}\G\big),
\]
where the last equality holds if $\M$ is invertible and $\det\M=1$, as is the case for $\F$ and $\B$.
Considering that $\dvg\vc u=\tr(\nabla\vc u)=0$, we immediately obtain
\begin{equation}\label{eq:othogonality}
\langle\delta\mathcal C[\B],\nabla\vc u\B+\B\nabla\vc u^\tsp\rangle = \dint\tr\big(\B^{-1}(\nabla\vc u\B+\B\nabla\vc u^\tsp)\big)=\dint 2\tr(\nabla\vc u)=0,
\end{equation}
proving that, for any $\B$ in the appropriate set, the right-hand side of \eqref{eq:Bonly} is orthogonal to the constraint normal $\delta\mathcal C[\B]$ in the sense of the tensor scalar product in $L^2([0,T]\times\Omega;\mat_d(\bR))$, induced by the matrix product $\vt A:\vt C\coloneqq\tr(\vt A\vt C^\tsp)$.
The transpose on the second factor is irrelevant whenever $\vt A$ or $\vt C$ is symmetric.

\subsection{Viscoelastic Cauchy stress}

The evolution equation corresponding to the local balance of linear momentum for a continuum in the Eulerian setting takes the form
\begin{equation}\label{eq:flow_general}
\rho \adv{u}\uu=\dvg\tT+\rho\vc f,
\end{equation}
where $\rho$ is the mass density, $\vc u$ is the velocity field, $\rho\vc f$ is a given force density, and $\tT$ is the Cauchy stress tensor.
The description of specific materials is addressed by prescribing a constitutive law that expresses the dependence of the stress $\tT$ upon kinematic quantities.

For an incompressible viscoelastic material the Cauchy stress can be additively decomposed in three terms as $\tT = -p\vt{I} + \tT_\mathrm{vi} + \tT_\mathrm{el}$. 
The pressure field $p$ takes the role of a Lagrange multiplier for the incompressibility constraint and the corresponding isotropic pressure term adsorbs all the spherical part of the stress, implying that suitable forms of the viscous stress $\tT_\mathrm{vi}$ and of the elastic contribution $\tT_\mathrm{el}$ should be traceless.
For the viscous stress, we assume the simple Newtonian form $\tT_\mathrm{vi}=2\eta\vt D$, where $\eta$ is a constant viscosity and $\vt D=\tfrac{1}{2}(\nabla\vc u+\nabla\vc u^\tsp)$ is the symmetric deformation rate tensor.
Note that, for an incompressible material, $\tr\vt D=\tr\nabla\vc u=0$.

As for the elastic stress, it must depend on a measure of the deformation that neglects rigid rotations, such as the left Cauchy--Green tensor $\B$. 
As argued above, since the evolution driven by $\tT_\mathrm{el}$ should be tangent to the manifold of deformation tensors with unit determinant, a natural choice is to consider the matrix logarithm $\log\B$ (a spatial version of the Hencky strain tensor), which is traceless whenever $\det\B=1$.
We thus assume $\tT_\mathrm{el}=\kappa(\log\B-\log\Bref)$, with $\kappa>0$ being an elasticity constant and $\Bref$ representing the relaxed state of deformation, in which no elastic stress arises.
For the time being, we can set $\Bref=\vt I$ (as customary in solid mechanics) and obtain the following form of the constitutive prescription for a viscoelastic solid material:
\begin{equation}\label{eq:T_visco_solid}
\tT=-p\vt I+2\eta\vt D+\kappa\log\B.
\end{equation}

By choosing a fixed $\Bref$ we are stating that the relaxed configuration never changes and, in this precise sense, we can say that the viscoelastic material we are representing is a solid. We will introduce in Section~\ref{sec:fluid_model} a viscoelastic fluid model, in which the relaxed configuration represented by $\Bref$ evolves over time as a consequence of plastic relaxation phenomena.

In summary, the coupled equations \eqref{eq:Bonly} and \eqref{eq:flow_general}, with the definition \eqref{eq:T_visco_solid} and suitable initial and boundary conditions, represent the differential problem that describes the incompressible evolution of our viscoelastic solid.

%%%%%%%%%%%%%%%%%%%%%%%%%%%%%%%
\section{Tensorial transport equations and charted weak solutions}\label{sec:equations_and_solutions}

Here we introduce a suitable notion of solution for tensorial transport equations that will lead to a satisfactory treatment, in Section~\ref{sec:existence_solid}, of the coupled system of evolution equations for the viscoelastic solid model obtained by \eqref{eq:Bonly} and \eqref{eq:flow_general} under the constitutive assumption \eqref{eq:T_visco_solid} and the incompressibility constraint.
We will discuss in Section~\ref{sec:further_applications} the broader scope of applicability our approach in the context of continuum mechanical models.

If we are given a smooth divergence-free vector field $\vc u$, we can consider the Cauchy problem
\begin{equation}\label{eq:F}
\begin{cases}
\partial_t \F+{\vc u}\cdot\nabla \F=\nabla {\vc u} \F,\\
\F(0,\cdot)=\F_0,
\end{cases}
\end{equation}
where $\F_0$ is some smooth matrix-valued function with $\det\F_0=1$. 
Then, by setting $\B:=\F\F^\tsp$ and $\B_0:=\F_0\F_0^\tsp$, we have a solution of
\begin{equation}\label{eq:B}
\begin{cases}
\partial_t\B+\vc u\cdot\nabla\B=\nabla \vc u\B+\B(\nabla \vc u)^\tsp,\\
\B(0,\cdot)=\B_0.
\end{cases}
\end{equation}

If the vector field is irregular, namely $\vc u\in L^1((0,T);W^{1,p}(\Omega))$ for some $1<p<\infty$, the Cauchy--Lipschitz theory cannot be applied and, in particular, distributional solutions of \eqref{eq:B} may not be defined. This is indeed our context: if $\vc u$ is the velocity field which solves a Navier--Stokes-type equation, we have that $\vc u$ is divergence-free and typically $\vc u\in L^\infty((0,T);L^2(\Omega))\cap L^2((0,T);H^1_0(\Omega))$.

\subsection{A Hilbert manifold of tensor fields}

To describe divergence-free velocity fields in a weak sense, we employ the standard spaces $H$ and $V$ defined as the closure of smooth compactly supported divergence-free vector fields on $\Omega$ with respect to the norm in $L^2(\Omega;\bR^d)$ and $H^1(\Omega;\bR^d)$, respectively.
In particular, $V$ is the subspace of $H^1_0(\Omega;\bR^d)$ of weakly divergence-free vector fields vanishing on $\partial \Omega$ in the sense of traces.

To appropriately deal with the relevant tensor fields, we introduce the set of symmetric traceless $d\times d$ matrices $\MM\coloneqq\{\M\in\mat_d(\bR):\M^\tsp=\M,\tr\M=0\}$ and the set of tensor fields
\begin{equation}
    \BB\coloneqq\left\{\B_0:\Omega\to\mat_3(\bR): \B_0=\B_0^\tsp,\;\det\B_0=1,\text{ and }\log\B_0\in L^2(\Omega;\MM)\right\}.
\end{equation}
We will prove that, for any divergence-free velocity field $\vc u\in L^2([0,T];V)$, equation \eqref{eq:Bonly} admits a solution in the set
\[
\BB_T\coloneqq\left\{\begin{matrix}
\B:[0,T]\times\Omega\to\mat_3(\bR)\text{ such that } \B=\B^\tsp,\;\det\B=1,\\
\phantom{\int}\text{ and }\log\B\in L^2([0,T]\times\Omega;\MM)\cap L^\infty([0,T];L^2(\Omega;\MM))\end{matrix}\right\}
\]
for any $T>0$ and for any choice of the initial condition $\B_0\in \BB$.

Note that, while $\MM$ is a linear space, neither $\BB$ not $\BB_T$ are such, due to the nonlinear constraint of unit determinant.
Most importantly, it is the logarithm of elements of $\BB$ and $\BB_T$ that belongs to a linear space.
The set $\BB_T$ is a Hilbert manifold. In fact, the matrix logarithm is bijective and differentiable on symmetric positive definite tensor fields, and provides an atlas covering $\BB_T$ with a single local chart represented by the Hilbert space $L^2([0,T]\times\Omega;\MM)$.
This fact allows us to introduce a \emph{charted weak topology} on $\BB_T$ such that a sequence $\{\B_h\}$ converges weakly to $\B$ if and only if the sequence $\{\log\B_h\}$ converges weakly to $\log\B$ in $L^2([0,T]\times\Omega;\MM)$.

\subsection{Charted weak solutions}

Based on the notion of charted weak convergence in the set $\BB_T$ mentioned above, we can give our main definition and results. In what follows, we will denote the norm on spaces of the form $L^p([0,T];X)$ by $\|\cdot\|_{L^pX}$.

\begin{definition}\label{def:charted}
Given $\vc u\in L^\infty([0,T];H)\cap L^2([0,T];V)$ and $\B_0\in\BB$, we say that $\B\in \BB_T$ is a \emph{charted weak solution} of the transport equation \eqref{eq:B} with initial datum $\B_0$ if
there exist two sequences $\{\vc u_k\}$ and $\{\log\B_{0,k}\}$ of smooth fields that satisfy
\begin{itemize}
    \item[(i)] $\uu_k\weaktos\uu$ in $L^\infty([0,T];H)\cap L^2([0,T];V)$,
    \item[(ii)] $\log \B_{0,k}\weakto\log\B_0$ in $L^2(\Omega;\MM)$,
\end{itemize}
and such that the corresponding sequence of smooth solutions $\{\B_k\}$ of \eqref{eq:B} with advecting field $\vc u_k$ and initial condition $\B_{0,k}$ satisfies
\[
\log\B_k\weaktos\log \B \hspace{0.4cm}\mbox{in }L^\infty([0,T];L^2(\Omega;\MM)).
\]
In particular, $\B$ is the limit of $\{\B_k\}$ in $\BB_T$ with respect to the charted weak topology.
\end{definition}

We can now prove the main theorem of this section.
\begin{thm}\label{lem:esistenza-charted}
For any $\B_0\in\BB$ and $\vc u\in L^\infty([0,T];H)\cap L^2([0,T];V)$ there exists a charted weak solution $\B\in \BB_T$ of the Cauchy problem \eqref{eq:B}, which satisfies
\begin{equation}\label{est:norm_B}
\|\log\B\|_{L^\infty L^2}^2\leq 16T\|\nabla\uu\|_{L^2L^2}^2+2\|\log\B_0\|_{L^2}^2.
\end{equation}
\end{thm}

\begin{proof}
By taking convolutions with standard mollifiers, we can construct sequences $\{\vc u_k\}$ and $\{\log\B_{0,k}\}$ of smooth fields such that $\uu_k\weaktos\uu$ in $L^\infty([0,T];H)\cap L^2([0,T];V)$ and $\log \B_{0,k}\to\log\B_0$ in $L^2(\Omega;\MM)$. To obtain $\B_{0,k}$ we regularize $\log\B_0$ by convolution and then take the matrix exponential. In this way, we directly have a uniform bound on $\|\log\B_{0,k}\|_{L^2}$ as needed.
Then, for any $k\in\mathbb N$, we consider the unique smooth solution $\B_k$ of
\begin{equation}\label{eq:Bk}
\de_t\B_k+(\vc u_k\cdot\nabla)\B_k=\nabla\vc u_k\B_k+\B_k\nabla\vc u_k^\tsp
\end{equation}
with initial condition $\B_{0,k}$. Note that, since $\vc u_k$ is smooth, existence and uniqueness of a smooth solution $\B_k$ for problem \eqref{eq:B} is a classical result.

The central step of the proof is establishing suitable \emph{a priori} estimates on solutions of \eqref{eq:Bk}. We multiply equation \eqref{eq:Bk} by $\B^{-1}_k\log\B_k$ and, by integrating over $\Omega$, we obtain (see \eqref{eq:log_derivative} in Appendix~\ref{app:estimates})
\begin{multline}
\frac{\dee}{\dee t}\Vert\log\B_k\Vert^2_{L^2}+\int_\Omega(\vc u_k\cdot\nabla)(\log\B_k:\log\B_k)\dee x=4\int_\Omega(\nabla\vc u_k:\log\B_k)\dee x\\
\leq \frac{2}{\varepsilon}\Vert\nabla\vc u_k\Vert^2_{L^2}+2\varepsilon\Vert\log\B_k\Vert^2_{L^2},
\end{multline}
where the last term comes from applying Young's inequality with a parameter $\varepsilon>0$, with units of inverse time, that will be chosen later.
The second term of the left-hand side vanishes since $\vc u_k$ vanishes on $\de\Omega$ and $\dvg\, \vc u_k=0$.
Hence, by integrating in time, we obtain
\begin{equation}\label{eq:1int}
\Vert\log\B_k(t,\cdot)\Vert^2_{L^2}\leq \frac{2}{\varepsilon}\Vert\nabla\vc u_k\Vert^2_{L^2L^2}+2\varepsilon\Vert\log\B_k\Vert^2_{L^2L^2}
+\Vert\log\B_{0,k}\Vert^2_{L^2}
\end{equation}
and by the trivial bound
$$
\|\log\B_k\|^2_{L^2L^2}\leq T\|\log\B_k\|^2_{L^\infty L^2},
$$
we obtain that
\begin{align}
(1-2\varepsilon T)\Vert\log\B_k\Vert^2_{L^\infty L^2}&\leq \frac{2}{\varepsilon}\Vert\nabla\vc u_k\Vert^2_{L^2L^2}+\Vert\log\B_{0,k}\Vert^2_{L^2}\label{eq:2int}\nonumber\\
&\leq \frac{2}{\varepsilon}\|\nabla\uu\|_{L^2L^2}^2+\|\log\B_0\|_{L^2}^2,
\end{align}
where in the last step we used the properties of the mollifiers.
By choosing $\varepsilon<1/(2T)$, say $\varepsilon=1/(4T)$, from the combination of \eqref{eq:1int} and \eqref{eq:2int} we conclude that any solution of \eqref{eq:Bk} belongs to $\BB_T$, 
since $\log\B_k\in L^2([0,T]\times\Omega;\MM)\cap L^\infty([0,T];L^2(\Omega;\MM))$, and the estimate \eqref{est:norm_B} is satisfied by each $\B_k$.

Moreover, the sequence $\{ \log\B_k \}$ is uniformly bounded in the spaces $L^2([0,T]\times\Omega;\MM)$ and $L^\infty([0,T];L^2(\Omega;\MM))$. Hence, the existence of a charted weak solution $\B$ for \eqref{eq:B} follows from a standard compactness argument.
The properties of being symmetric and traceless are linear constraints defining closed subspaces and pass naturally to the limit. This implies that $\B$ is symmetric and $\det\B=1$.
\end{proof}

\begin{rem}
With Theorem~\ref{lem:esistenza-charted}, we have proven that the approximation procedure described in the definition of charted weak solutions admits a limit which is well behaved, in the sense that it satisfies the integrability properties that we can \emph{a priori} establish for a solution of the tensorial transport equation with natural assumptions on the regularity of the initial data and the advecting velocity field.
\end{rem}

\begin{rem}
A better characterization for charted weak solutions remains elusive. This is partly due to the fact that we cannot pass to the limit directly in the distributional formulation of the equations. 
A major obstacle originates from the fact that derivatives of a tensor field do not in general commute with the tensor itself (at variance with the scalar case) and a theory of renormalized solutions (in the spirit of \cite{CCS} for example) does not seems to work. 
This is intimately linked to the nonlinear nature of the tensor manifold that becomes apparent when we consider the evolution equation for the logarithm itself, as shown in Appendix~\ref{app:log-eq}.

\end{rem}

%%%%%%%%%%%%%%%%%%%%%%%%%%%%%%%
\section{Existence of solutions for the viscoelastic solid}
\label{sec:existence_solid}

By substituting the constitutive assumption \eqref{eq:T_visco_solid} in the balance of linear momentum \eqref{eq:flow_general} and neglecting external forces ($\vc f=\vc 0$) for simplicity, we obtain the equations for our viscoelastic solid. %
Combining it with the incompressibility constraint, the evolution equations for the left Cauchy--Green tensor, and with suitable initial and boundary conditions, we obtain, in a time-space domain $[0,T]\times\Omega$, the differential problem
\begin{subequations}
\begin{align}
\rho \adv{u}\vc u&=-\nabla p +\eta\Delta\vc u+\kappa\dvg\log\B,\label{eq:flow_solid}\\
\adv{u}\B&=\nabla\vc u\B+\B\nabla\vc u^\tsp,\label{eq:transport_B}\\
\dvg\vc u&=0,\;
\det\B=1,\\
\vc u(0,\cdot)&=\vc u_0,\;
\B(0,\cdot)=\B_0,\\
\vc u\vert_{\de\Omega}&=\vc 0.
\end{align}\label{eq:sistema-logvisc}
\end{subequations}

\begin{definition}\label{def:weak-solution}
We say that a pair $(\uu,\B)$ is a {\em Leray weak solution} of the differential problem \eqref{eq:sistema-logvisc} if
\begin{itemize}
    \item[(i)] $\uu\in L^\infty([0,T];H)\cap L^2([0,T];V)$ satisfies equation \eqref{eq:flow_solid} in the distributional sense, namely
    \begin{multline}
        \dint\rho\left( \uu\,\de_t\vc\Theta+\uu\otimes\uu:\nabla\vc\Theta\right)\dee x\dee t+\int_\Omega \rho\uu_0\vc\Theta(0,\cdot)\dee x\\
        =-\dint\left(2\eta\vt D+\kappa\log\B\right):\nabla\vc\Theta\,\dee x \dee t,
    \end{multline}
    for all divergence-free test functions $\vc\Theta\in C^\infty_c([0,T)\times\Omega)$;
    \item[(ii)] $\B\in \BB_T$ is a charted weak solution of equation \eqref{eq:transport_B} with initial datum $\B_0$;
    \item[(iii)] the pair $(\uu,\B)$ satisfies, for almost every $t\in[0,T]$, the energy inequality
    \begin{multline}\label{est:energia}
        \rho\|\uu(t,\cdot)\|_{L^2}^2+\frac{\kappa}{2}\|\log\B(t,\cdot)\|_{L^2}^2+4\eta\int_0^t\|\vt D(s,\cdot)\|_{L^2}^2\,\dee s\\
        \leq \rho\|\uu_0\|_{L^2}^2+\frac{\kappa}{2}\|\log\B_0\|_{L^2}^2.
    \end{multline}
\end{itemize}
\end{definition}

Note that the energy inequality guarantees that our model is thermodynamically consistent, because the sum of kinetic, elastic, and dissipated energy is bounded by the total energy at the initial time.
To prove the existence of Leray weak solutions for the differential problem \eqref{eq:sistema-logvisc}, we will make use of the following theorems (see~\cite[Corollary 8.1]{Deimling_1985} and \cite{Lions_1978}, respectively).

\begin{thm}[Leray--Schauder]\label{thm:leray-schauder}
Let $(X, \| \cdot\|)$ be a Banach space and let $A: X \to X$ be a compact operator. Then, either $Au=u$ has a solution or the set
$$
S\coloneqq\{u\in X: Au=\lambda u \mbox{ for some }\lambda>1 \},
$$
is unbounded.
\end{thm}

\begin{lem}\label{lem:lions}
Let $\Omega\subset \bR^d$ be an open and bounded domain with Lipschitz boundary and let $\vc v_n$ be a bounded sequence in $L^\infty([0,T];L^2(\Omega))\cap L^2([0,T];H^1_0(\Omega))$. Assume that there exists $C$ and $\alpha>0$ such that, for all $\delta\geq 0$ sufficiently small,
\begin{equation}\label{stima:lions}
\int_0^{T-\delta} \|\vc v_n(t+\delta,\cdot)-\vc v_n(t,\cdot)\|_{L^2}^2\dee t\leq C\delta^\alpha.
\end{equation}
Then, the sequence admits a convergent subsequence in $L^2([0,T];L^2(\Omega))$.
\end{lem}
\begin{proof}
We denote by $\tilde{\vc v}_n$ the sequence defined on $\bR^{d+1}$ which is equal to $\vc v_n$ on $(0,T)\times \Omega$ and $0$ otherwise. Let $0<\delta<T$, we have that
\begin{align*}
    \int_{-\infty}^{+\infty}\int_{\bR^d}\vert\tilde{\vc v}_n(t+\delta,x)-\tilde{\vc v}_n(t,x)\vert^2\dee x\dee t&=\int_{0}^{T-\delta}\int_{\Omega}\vert\vc v_n(t+\delta,x)-\vc v_n(t,x)\vert^2\dee x\dee t\\
    &+\int_{T-\delta}^T\int_{\bR^d}\vert\tilde{\vc v}_n(t+\delta,x)-\vc v_n(t,x)\vert^2\dee x\dee t\\
    &\leq C\delta^\alpha+2\delta\|\vc v_n\|_{L^\infty L^2}^2.
\end{align*}
On the other hand, if $h\in\bR^d$, we have that
\begin{align*}
    \int_{-\infty}^{+\infty}\int_{\bR^d}\vert\tilde{\vc v}_n(t,x+h)-\tilde{\vc v}_n(t,x)\vert^2\dee x\dee t&=\int_{0}^{T}\int_{\bR^d}\left\vert\int_0^1\nabla\tilde{\vc v}_n(t,x+sh)h\dee s\right\vert^2\dee x\dee t\\
    &\leq C\vert h\vert^2\|\nabla \vc v_n\|_{L^2 L^2}^2,
\end{align*}
where in the last line we used Sobolev's extension theorem. Then, since by assumption the norms $\|\vc v_n\|_{L^\infty L^2}$ and $\|\nabla \vc v_n\|_{L^2 L^2}$ are equibounded, the assertion follows from an application of Riesz--Fr\'echet--Kolmogorov theorem.
\end{proof}

We can now prove the main theorem of this section.
\begin{thm}\label{thm:main1}
Given $T>0$, for any $\uu_0\in H$ and $\B_0\in \BB$, there exists a Leray weak solution $(\uu,\B)$ of the differential problem \eqref{eq:sistema-logvisc} in the sense of Definition \ref{def:weak-solution}.
\end{thm}
\begin{proof}
We divide the proof in several steps.\\
\\
\underline{\em Step 1} \hspace{0.3cm} {\bf Construction of the approximating sequence (I).} We use a Galerkin scheme. Let $V_1\subset V_2\subset \cdots \subset H^1_0(\Omega)$ be a sequence of spaces of of smooth divergence-free functions, and let $\bigcup_n V_n$ be dense in $V\coloneqq\{\vc v\in H^1_0(\Omega):\dvg \vc v=0 \}$. We assume that $V_n=\mathrm{span}(\vc \phi_1,...,\vc \phi_n)$ where $\vc\phi_i$ is an eigenfunction of the Stokes system, i.e. they satisfy
\begin{equation}
    \begin{cases}
    -\eta \Delta \vc\phi_i+\nabla\pi_i=\lambda_i\vc\phi_i,\hspace{0.3cm} &\mbox{in }\Omega,\\
    \dvg\vc\phi_i=0,\hspace{1cm} &\mbox{in }\Omega,\\
    \vc\phi_i=0,\hspace{1.2cm} &\mbox{on }\partial\Omega,
    \end{cases}
\end{equation}
where $\{\lambda_i\}$ is a countable non-decreasing positive sequence. Note that functions in $V_n$ are smooth. For a given $\vv\in C([0,T];H^1_0(\Omega))$, we define $\vc v_n=\Pi_n\vv\in C([0,T]; V_n)$, where $\Pi_n:L^2(\Omega)\to V_n$ is the standard projector. Note that, since divergence-free vector fields are orthogonal to gradients, the pressure term disappears when projecting the flow equation. Finally, consider the following linear system
\begin{equation}\label{eq:sistema_approssimante_lineare}
    \begin{cases}
    \rho\partial_t\vc u_n+\rho\Pi_n\left[(\vc v_n\cdot\nabla)\vc u_n\right]=\eta\Pi_n\Delta\vc u_n+\kappa\Pi_n\dvg\log\B_n,\\
    \dvg \uu_n=0,\\
    \partial_t\B_n+(\vc v_n\cdot\nabla)\B_n=\nabla \vc v_n\B_n+\B_n(\nabla \vc v_n)^\tsp,\\
	\uu_n(0,\cdot)=\Pi_n \uu_0,\hspace{0.3cm}\B_n(0,\cdot)=\exp[(\log\B_{0})_n].
    \end{cases}
\end{equation}
Since $\vv_n$ is smooth, we can infer that for every $n$ there exists a unique smooth solution $\B_n$ of the third equation in \eqref{eq:sistema_approssimante_lineare}. On the other hand, the first equation in \eqref{eq:sistema_approssimante_lineare} reduces to the system of ODEs with Lipschitz right-hand side given by
\begin{equation}\label{eq:ODE1}
    \rho\dot{c}_i^n(t)=-\lambda_i^n c_i^n(t)-\rho\sum_{j,k=1}^n B_{ijk} a_j^n(t) c_k^n(t)+b_i^n(t),\hspace{0.5cm}\mbox{for all }i=1,...n,
\end{equation}
where, denoting by $\langle\cdot,\cdot\rangle$ the $L^2$ scalar product, we used
\begin{align*}
    \uu_n&=\sum_{i=1}^n c_i^n(t)\vc\phi_i(x),\hspace{0.7cm}
    \vv_n= \sum_{i=1}^n a_i^n(t)\vc\phi_i(x),\\
    \kappa\Pi_n\dvg\log\B_n &= \sum_{i=1}^n b_i^n(t)\vc \phi_i(x),\hspace{0.5cm}
    B_{ijk}=\langle(\vc\phi_j\cdot\nabla)\vc\phi_k,\vc\phi_i\rangle.
\end{align*}
Then, by standard Cauchy--Lipschitz theory we can infer that there exists a smooth solution $\uu_n$ on some time interval $[0,T_n)$. It is worth noticing that, at this stage, we only have local existence because the Lipschitz constant of the right-hand side of \eqref{eq:ODE1} is not uniformly bounded with respect to $n$.\\
\\
\underline{\em Step 2} \hspace{0.3cm} {\bf Energy estimates on $\uu_n$.} We now show that the solution actually exists on any arbitrary set of times $[0,T]$. This is achieved by showing that the coefficients $c_i^n$ do not blow up in finite time. Let $s\in (0,T_n)$ and take the inner product of the first equation in \eqref{eq:sistema_approssimante_lineare} with $\uu_n(s)$:
\begin{multline}
    \rho\langle\partial_t\uu_n(s),\uu_n(s)\rangle+\eta\langle-\Delta \uu_n(s),\uu_n(s)\rangle+\rho\langle\Pi_n\left((\vv_n(s)\cdot)\uu_n(s)\right),\uu_n(s)\rangle\\
    =\kappa\langle\Pi_n\dvg\log\B_n(s),\uu_n(s)\rangle.
\end{multline}
It is easy to show, by integration by parts and application of Cauchy's and Young's inequalities, that
\begin{itemize}
    \item $\langle\partial_t\uu_n(s),\uu_n(s)\rangle=\frac{1}{2}\frac{\dee}{\dee t} \|\uu_n(s)\|^2_{L^2}$,
    \item $\langle  -\Delta\uu_n(s),\uu_n(s)\rangle=\|\nabla\uu_n(s)\|^2_{L^2}$,
    \item $\langle\Pi_n\left((\vv_n(s)\cdot)\uu_n(s)\right),\uu_n(s)\rangle=0$,
    \item $\langle\Pi_n\dvg\log\B_n(s),\uu_n(s)\rangle\leq \frac{\kappa}{2\eta}\|\log\B_n(s)\|^2_{L^2}+\frac{\eta}{2\kappa}\|\nabla\uu_n(s)\|^2_{L^2}$.
\end{itemize}
Then, by integrating in time, we get
\begin{multline*}
    \rho\|\uu_n(t,\cdot)\|^2_{L^2}+\eta\int_0^t\|\nabla\uu_n(s,\cdot)\|^2_{L^2}\dee s\leq \rho\|\Pi_n\uu_0\|^2_{L^2}+\frac{\kappa^2}{\eta}\int_0^t\|\log\B_n(s,\cdot)\|^2_{L^2}\dee s\\
    \leq \rho\|\uu_0\|^2_{L^2}+\frac{\kappa^2}{\eta}(16T\|\nabla\vv_n\|^2_{L^2L^2}+2\|\log\B_0\|^2_{L^2}),
\end{multline*}
for all $0\leq t< T_n<T$, where in the last inequality we have used Theorem \ref{lem:esistenza-charted}. Finally, since $\|\nabla\vv_n\|_{L^2L^2}$ is uniformly bounded and $\|\uu_n(t)\|^2_{L^2}=\sum_{i=1}^n\vert c_i^n(t)\vert^2$, it follows that the sequence $\{c_i^n\}_{i=1}^n$ does not blow-up in finite time.\\
\\
\underline{\em Step 3} \hspace{0.3cm} {\bf Construction of the approximating sequence (II).} Let $n\in\mathbb{N}$ be fixed and define the following operator
\begin{equation}
    \mathcal{F}: \vv_n\in C([0,T];V_n)\mapsto \uu_n \in C([0,T];V_n).
\end{equation}
Our goal is to show that $\mathcal{F}$ has a fixed point. To this end we want to apply Theorem \ref{thm:leray-schauder}; then, it is enough to show
\begin{itemize}
\item[(i)] $ \uu_n$ is equicontinuous in $C([0,T];V_n)$;
\item[(ii)] the set 
$
S\coloneqq\{\vv\in C([0,T];V_n):\mathcal{F}(\vv)=\lambda \vv \text{ for some }\lambda>1\}$
is bounded.
\end{itemize}
In particular, note that from (i) we immediately obtain that the operator $\mathcal{F}$ is completely continuous by applying Ascoli--Arzel\`a's Theorem. We start by proving (i), for which it is useful to show that $\partial_t\uu_n\in L^{4/3}([0,T];V^*)$. By setting $\uu_n=\mathcal{F}(\vv_n)$ and testing the equation with $\varphi\in V$, by Holder's inequality we get
\begin{multline*}
    \langle\partial_t\uu_n,\varphi\rangle= -\frac{\eta}{\rho}\langle\nabla\uu_n,\nabla\varphi_n\rangle-\frac{\kappa}{\rho}\langle\log\B_n,\nabla\varphi_n\rangle-\langle(\vv_n\cdot\nabla)\uu_n,\varphi_n\rangle\\
    \leq \left(\frac{\eta}{\rho}\|\nabla\uu_n(t)\|_{L^2}+\frac{\kappa}{\rho}\|\log\B_n(t)\|_{L^2}\right)\|\nabla\varphi\|_{L^2}+\|\vv_n(t)\|_{L^3}\|\nabla\uu_n(t)\|_{L^2}\|\varphi\|_{L^6},\label{stima-de-t}
\end{multline*}
where $\varphi_n\coloneqq\Pi_n\varphi$. By Sobolev's embedding we know that
$$
\|\varphi\|_{L^6}\leq C_1\|\nabla\varphi\|_{L^2},
$$
and then from \eqref{stima-de-t} it follows that
\begin{multline}
    \int_0^T\|\partial_t\uu_n(t)\|_{V^*}^{4/3}\dee t\leq \left(\frac{\eta}{\rho}\right)^{4/3}\int_0^T\|\nabla\uu_n\|_{L^2}^{4/3}\dee t\\+T\left(\frac{\kappa}{\rho}\right)^{4/3}\|\log\B_n\|^{4/3}_{L^\infty L^2})+C_1^{4/3}\int_0^T (\|\vv_n(t)\|_{L^3}^{4/3}\|\nabla\uu_n(t)\|_{L^2}^{4/3})\dee t\\
    \leq T^{1/3}\left(\frac{\eta}{\rho}\right)^{4/3}\|\nabla\uu_n\|_{L^2L^2}^{4/3}+T\left(\frac{\kappa}{\rho}\right)^{4/3}\|\log\B_n\|^{4/3}_{L^\infty L^2})\\
    +C_1^{4/3}\int_0^T (\|\vv_n(t)\|_{L^3}^{4/3}\|\nabla\uu_n(t)\|_{L^2}^{4/3})\dee t
    .
\end{multline}
Note that we can bound
$$
\int_0^T \left(\|\vv_n(t)\|_{L^3}^{4/3}\|\nabla\uu_n(t)\|_{L^2}^{4/3}\right)\dee t\leq \left(\int_0^T\|\vv_n(t)\|_{L^3}^4\dee t\right)^{\frac13}\left(\int_0^T\|\nabla\uu_n(t)\|_{L^2}^2\dee t\right)^{\frac23},
$$
and, by arguing again on Sobolev's embeddings, we have
\[
\int_0^T\|\vv_n(t)\|_{L^3}^4\dee t\leq \int_0^T C_2^2 \|\vv_n(t)\|^2_{L^2}\|\nabla\vv_n(t)\|^2_{L^2}\dee t\leq C_2^2 \|\vv_n\|^2_{L^\infty L^2}\|\nabla\vv_n\|_{L^2 L^2}^2.
\]
Finally, thanks to the energy estimates proved in {\em Step 2} we obtain
\begin{equation}\label{stima:det-un}
    \int_0^T\|\partial_t\uu_n(t)\|_{V^*}^{4/3}\dee t\leq C_3(T,\|\uu_0\|_{L^2},\|\log\B_0\|_{L^2},\|\vv_n\|_{L^\infty H\cap L^2 V}).
\end{equation}
Then, since on finite dimensional spaces all the norms are equivalent, we can write
$$
\|\uu_n(t_2,\cdot)-\uu_n(t_1,\cdot)\|_{H^1}\leq C_4\int_{t_1}^{t_2}\|\partial_t\uu_n(s,\cdot)\|_{V^*}\dee s\leq C_5\sqrt[4]{t_2-t_1},
$$
which proves (i). We now prove (ii): let $\vv_n\in S$, i.e. $\mathcal{F}(\vv_n)=\lambda \vv_n$ for some (fixed) $\lambda>1$. This means that $\vv_n$ satisfies 
\begin{equation}\label{eq:lambdav}
\begin{cases}
    \lambda\rho\partial_t \vv_n+\lambda\rho\Pi_n[(\vv_n\cdot\nabla)\vv_n]=\lambda\eta\Pi_n\Delta\vv_n+\kappa\Pi_n\dvg\log\B_n,\\
    \lambda\dvg \vv_n=0,\\
    \partial_t\B_n+(\vc v_n\cdot\nabla)\B_n=\nabla \vc v_n\B_n+\B_n(\nabla \vc v_n)^\tsp,\\
	\lambda\vv_n(0,\cdot)=\Pi_n \uu_0,\hspace{0.3cm}\B_n(0,\cdot)=\exp[(\log\B_{0})_n].
\end{cases}
\end{equation}
We multiply the first equation in \eqref{eq:lambdav} by $\vv_n$ and integrate in space and time to obtain
\begin{multline}\label{bilancio-vn}
    \lambda\rho\int_{\Omega}\vert\vv_n(t,x)\vert^2\dee x+2\lambda\eta\int_0^t\int_{\Omega}\vert\nabla \vv_n(s,x)\vert^2\dee x\dee s\\
    =\lambda\rho\int_{\Omega}\vert\Pi_n\uu_0(x)\vert^2\dee x-2\kappa\int_0^t\int_{\Omega}\log\B_n(s,x):\nabla\vv_n(s,x)\dee x\dee s
\end{multline}
On the other hand, by multiplying the equation for $\B_n$ by $\B^{-1}_n\log\B_n$ and arguing as in Theorem \ref{lem:esistenza-charted} we obtain
\begin{equation}\label{bilancio-Bn}
    \frac12\int_{\Omega}\vert\log\B_n(t,x)\vert^2\dee x=\frac12\int_\Omega\vert\log\B_{0,n}\vert^2\dee x+2\int_0^t\int_\Omega \nabla\vv_n:\log\B_n\dee x \dee s.
\end{equation}
Summing \eqref{bilancio-vn} plus $\kappa$ times \eqref{bilancio-Bn} we easily obtain
\begin{multline}
    \int_{\Omega}\vert\vv_n(t,x)\vert^2\dee x+\frac{2\eta}{\rho}\int_0^t\int_{\Omega}\vert\nabla \vv_n(s,x)\vert^2\dee x\dee s\\
    \leq \int_{\Omega}\vert\Pi_n\uu_0(x)\vert^2\dee x+\frac{1}{\lambda}\frac{\kappa}{2\rho}\int_\Omega\vert\log\B_{0,n}\vert^2\dee x.
\end{multline}
Considering that $\lambda > 1$, we can show that $\|\vv_n\|_{L^\infty L^2}$ and $\|\vv_n\|_{L^2 H^1}$ are bounded uniformly with respect to $\lambda$.
Since all the norms are equivalent on finite-dimensional subspaces, this implies  the boundedness of the set $S$.
In conclusion, we have proven that $\mathcal{F}$ satisfies the hypothesis of Theorem \ref{thm:leray-schauder}, hence it follows that there exists a fixed point such that $\uu_n=\mathcal{F}(\uu_n)$.
\\
\\
\underline{\em Step 4} \hspace{0.3cm} {\bf Convergence of the approximating sequence.} Let $\{(\uu_n,\B_n)\}$ be a sequence with $\uu_n$ fixed points as constructed in previous step. We show that such a sequence converges to a Leray weak solution of \eqref{eq:sistema-logvisc}. First of all, by the energy estimates proved in {\em Step 2}, the sequence $\{\uu_n\}$ satisfies 
the bound 
$$
\|\uu_n\|_{L^\infty L^2\cap L^2 H^1}\leq C(\uu_0,\log\B_0).
$$
Then, there exists a vector field $\uu\in L^\infty L^2\cap L^2 H^1$ such that
$$
\uu_n\weakto\uu \;\;\mbox{in }L^\infty L^2\cap L^2 H^1.
$$
Moreover, following the proof of Theorem~\ref{lem:esistenza-charted}, we obtain the existence of a charted weak solution $\B$ for equation \eqref{eq:transport_B} such that the sequence approximating $\log\B$ is precisely $\{\log\B_n\}$.

The only thing left to prove is the compactness in $L^2([0,T];H)$ for the sequence $\{\uu_n\}$ in order to pass to the limit in the integral formulation of the equation. This follows from Lemma~\ref{lem:lions}: note that for a sequence of fixed points the constant in \eqref{stima:det-un} does not depend on $n$, telling us that an estimate like the one in \eqref{stima:lions} holds.
\\
\\
\underline{\em Step 5} \hspace{0.3cm}\textbf{Energy inequality.} 
We first observe that, for a smooth divergence-free vector field $\uu$ that vanishes on $\de\Omega$, we have
\begin{multline}\label{eq:DeqGrad}
\|\nabla\uu\|^2_{L^2}=-\int_\Omega\dvg\nabla\uu\cdot\uu=-\int_\Omega(\dvg\nabla\uu+\nabla\dvg\uu)\cdot\uu\dee x\\
=-2\int_\Omega\dvg\vt D\cdot\uu\dee x=2\int_\Omega\vt D:\nabla\uu\dee x=2\|\vt D\|^2_{L^2}.
\end{multline}
By arguing as in the proof of $(ii)$ in {\em Step 3}, $(\uu_n,\B_n)$ satisfies the energy balance
\begin{multline}
    \rho\|\uu_n(t,\cdot)\|_{L^2}^2+\frac{\kappa}{2}\|\log\B_n(t,\cdot)\|_{L^2}^2+4\eta\int_0^t\|\vt D_n(s,\cdot)\|_{L^2}^2\,\dee s\\
    = \rho\|\Pi_n\uu_0\|_{L^2}^2+\frac{\kappa}{2}\|\log\B_{0,n}\|_{L^2}^2.
\end{multline}
Then, by the lower semi-continuity of the norms with respect to weak convergence and the strong convergence of the initial datum, we get that $(\uu,\B)$ satisfies the energy inequality in Definition \ref{def:weak-solution}.
\end{proof}

%%%%%%%%%%%%%%%%%%%%%%%%%%%%%%%
\section{A viscoelastic fluid model}
\label{sec:fluid_model}

As mentioned above, we can construct models for viscoelastic fluids by keeping the very simple form $\tT_\mathrm{el}=\kappa(\log\B-\log\Bref)$  of the elastic stress, but letting $\Bref$ evolve and consequently depart from the identity tensor used for the solid model. The visoelastic stress tensor is again
\begin{equation}\label{eq:T_visco_fluid}
\tT=-p\vt I+2\eta\vt D+\kappa(\log\B-\log\Bref),
\end{equation}
but the prescription of an evolution equation for $\Bref$ is now necessary and represents the heart of the modeling effort.

We base our proposal of what could be regarded as the simplest example of a more general class of models on a few basic considerations.
First of all, the tensor $\Bref$ represents a state of incompressible deformation in which the elastic stress vanishes and, as such, we must keep $\det\Bref=1$ and the symmetry of $\Bref$ throughout the evolution.
This suggests that a general evolution equation for $\Bref$ should be of the form
\[
\adv{u}\Bref=\vt A\Bref+\Bref\vt A^\tsp
\]
for some traceless tensor field $\vt A$.
The fact that $\vt A$ should be traceless to preserve the unit determinant property is apparent from the discussion in Section~\ref{sec:incompressibility}.

Then, if we were to start from a situation in which $\Bref=\B$ and apply a rigid rotation to the material, no elastic stress should arise. It means that $\Bref$ must rotate with $\B$ and hence $\vt A$ should reduce to  $\W=\tfrac12(\nabla\uu-\nabla\uu^\tsp)$ in that special case.
More generally, we may wish to have the eigenvectors of $\Bref$ that always rotate coherently with those of $\B$. To this end, we can extract from $\nabla\uu$ a contribution that, under the evolution generated by equation~\eqref{eq:Bonly}, leaves the eignevalues of $\B$ unchanged, while inducing the rotation of its eigenvectors.
This can be achieved, as explained in Appendix~\ref{app:rotation}, by projecting out the component of $\nabla\uu$ on the subspace of symmetric tensors that commute with $\B$.
We denote by $\vt Q$ the component of $\nabla\uu$ orthogonal to that subspace, for which we obviously have the bound $\|\vt Q\|_{L^2}\leq\|\nabla\uu\|_{L^2}$.

Finally, we need to introduce a relaxation parameter, $\taur$, that sets a time scale over which, in a static deformation experiment, the elastic stress relaxes significantly.
It means that the eigenvalues of $\Bref$ should evolve as long as $\log\Bref$ is different from $\log\B$.
By combining these last requirements we can set $\vt A=\vt Q+\taur^{-1}(\log\B-\log\Bref)$ and obtain
\begin{equation}\label{eq:Bref_evolution}
\adv{u}\Bref=\vt Q\Bref+\Bref\vt Q^\tsp+ \frac{1}{\taur}\big[(\log\B-\log\Bref)\Bref+\Bref(\log\B-\log\Bref)\big].
\end{equation}
We stress that the model obtained combining \eqref{eq:T_visco_fluid} and \eqref{eq:Bref_evolution} is possibly the simplest model that satisfies the physical and mathematical considerations given above, but surely not the only possible choice.

The differential problem that describes the viscoelastic flow in $[0,T]\times\Omega$ thus becomes 
\begin{subequations}
\begin{align}
\rho \adv{u}\vc u&=-\nabla p +\eta\Delta\vc u+\kappa\dvg(\log\B-\log\Bref),\label{eq:flow_fluid}\\
\adv{u}\B&=\nabla\vc u\B+\B\nabla\vc u^\tsp,\label{eq:transport_B_2}\\
\adv{u}\Bref&=\vt Q\Bref+\Bref\vt Q^\tsp\notag\\
\phantom{\adv{u}\Bref=}&\quad+ \taur^{-1}\big[(\log\B-\log\Bref)\Bref+\Bref(\log\B-\log\Bref)\big],\label{eq:transport_Bref}\\
\dvg\vc u&=0,\;
\det\B=\det\Bref=1,\\
\vc u(0,\cdot)&=\vc u_0,\;
\B(0,\cdot)=\Bref(0,\cdot)=\B_0,\label{eq:initB}\\
\vc u\vert_{\de\Omega}&=\vc 0.
\end{align}\label{eq:sistema-fluido}
\end{subequations}

\subsection{A priori estimates}
We now prove a formal a priori estimate on $\Bref$ that plays a key role in what follows. Assume that $\Bref$ is a smooth tensor which satisfies \eqref{eq:transport_Bref}. We multiply equation \eqref{eq:transport_Bref} by $\Bref^{-1}\log\Bref$ and observe that
\begin{itemize}
    \item[(i)] $\de_t\Bref:\Bref^{-1}\log\Bref=\frac12\frac{\dee}{\dee t}\vert\log\Bref\vert^2$,
    \item[(ii)] $(\uu\cdot\nabla)\Bref:\Bref^{-1}\log\Bref=\frac12(\uu\cdot\nabla)\vert\log\Bref\vert^2$,
    \item[(iii)] $\vt Q\Bref:\Bref^{-1}\log\Bref=\tr(\vt Q^\tsp\log\Bref)=\vt Q:\log\Bref$,
    \item[(iv)] $\Bref \vt Q^\tsp:\Bref^{-1}\log\Bref=\vt Q:\log\Bref$,
    \item[(v)] $(\log \B-\log\Bref)\Bref:\Bref^{-1}\log\Bref=(\log \B-\log\Bref):\Bref$,
    \item[(vi)] $\Bref(\log \B-\log\Bref):\Bref^{-1}\log\Bref=(\log \B-\log\Bref):\Bref$.
\end{itemize}
Then, by using the identities above and integrating in space we obtain
\begin{multline}
    \frac{\dee}{\dee t}\|\log\Bref(t,\cdot)\|_{L^2}^2=4\int_\Omega \vt Q:\log \Bref \dee x\\
    +\frac{4}{\taur}\int_\Omega \log\B:\log\Bref \dee x-\frac{4}{\taur} \|\log\Bref\|^2_{L^2}.
\end{multline}
The last term on the right hand side above is negative, then we can drop it and integrate in time to get
\begin{multline}
    \|\log\Bref(t,\cdot)\|_{L^2}^2\leq 4\int_0^t\int_\Omega \vt Q:\log \Bref \dee x\dee s\\
    +\frac{4}{\taur}\int_0^t\int_\Omega \log\B:\log\Bref \dee x\dee s+\|\log\Bref(0,\cdot)\|_{L^2}^2,
\end{multline}
and, by using Young's inequality, we obtain
\begin{multline}
 \|\log\Bref\|_{L^\infty L^2}^2\leq 2\varepsilon_1\|\vt Q\|_{L^2L^2}^2+2T\left(\frac{1}{\varepsilon_1}
 +\frac{1}{\taur\varepsilon_2}\right)\|\log\Bref\|_{L^\infty L^2}^2\\
 +\frac{2\varepsilon_2T}{\taur}\|\log\B\|_{L^\infty L^2}^2+\|\log\B_0\|_{L^2}^2.
\end{multline}
By choosing $\varepsilon_1=8T$ and $\varepsilon_2=8T/\taur$ and since $\vt Q$ is a projection of $\nabla\uu$, we finally get
\begin{equation}\label{est:Bref}
    \|\log\Bref\|_{L^\infty L^2}^2\leq 32T\|\nabla\uu\|_{L^2L^2}^2+\frac{32T^2}{\taur^2}\|\log\B\|_{L^\infty L^2}^2+2\|\log\B_0\|_{L^2}^2.
\end{equation}

A second important relation is an identity involving the quantity $\log\B-\log\Bref$. In the model under consideration, thanks to the properties of the tensorial evolution equations, if $\B$ and $\Bref$ have the same eigenvectors at time $t=0$, then they commute at any time. Under such provisions, we have $\log\B-\log\Bref=\log(\B\Bref^{-1})$ and the tensor field $\B\Bref^{-1}$ is symmetric positive definite and with unit determinant. 
In light of the identity~\eqref{eq:inverse_derivative} presented in Appendix~\ref{app:estimates}, we have
\begin{equation}\label{eq:BBref}
\adv{u}(\B\Bref^{-1})=\nabla\uu\B\Bref^{-1}+\B\nabla\uu^\tsp\Bref^{-1}-\B\Bref^{-1}\vt A-\B\vt A^\tsp\Bref^{-1}.
\end{equation}
Multiplying equation \eqref{eq:BBref} by $(\Bref\B^{-1})\log(\B\Bref^{-1})$, taking into account that all tensors involved except $\nabla\uu$ and $\vt A$ are symmetric and commute with each other, and integrating in space, we obtain
\[
\frac{\dee}{\dee t}\|\log(\B\Bref^{-1})\|^2_{L^2}=4\int_\Omega(\nabla\uu-\vt A):\log(\B\Bref^{-1})\dee x.
\]
By considering that $\vt A=\vt Q+\taur^{-1}\log(\B\Bref^{-1})$ and that $\vt Q$ is orthogonal to symmetric tensors that commute with $\B$, we finally arrive at the identities
\begin{equation}
\frac{\dee}{\dee t}\|\log(\B\Bref^{-1})\|^2_{L^2}=4\int_\Omega\nabla\uu:\log(\B\Bref^{-1})\dee x-\frac{4}{\taur}\|\log(\B\Bref^{-1})\|^2_{L^2},\label{eq:Brefid1}
\end{equation}
and
\begin{multline}
\|\log(\B\Bref^{-1})(t,\cdot)\|^2_{L^2}\\=4\int_0^t\int_\Omega\vt D:\log(\B\Bref^{-1})\dee x\dee s-\frac{4}{\taur}\int_0^t\|\log(\B\Bref^{-1})\|^2_{L^2}\dee s,\label{eq:Brefid2}
\end{multline}
where we used the symmetry of $\log(\B\Bref^{-1})$ and the initial condition~\eqref{eq:initB}.

\subsection{Existence of Leray weak solutions}
In this section we prove an existence theorem of weak solutions for the system \eqref{eq:sistema-fluido}. 
In essence, we need to extend the notion of charted weak solution, to cover the equation for $\Bref$, and that of Leray weak solutions and then follow arguments analogous to those in the previous sections.

\begin{definition}\label{def:charted_extended}
Given $\vc u\in L^\infty([0,T];H)\cap L^2([0,T];V)$, $\B\in\BB_T$, and $\B_0\in\BB$, we say that $\Bref\in \BB_T$ is a \emph{charted weak solution} of the transport equation \eqref{eq:Bref_evolution} with initial datum $\B_0$ if
there exist sequences $\{\vc u_k\}$, $\{\log\B_{k}\}$, and $\{\log\B_{0,k}\}$ of smooth fields that satisfy
\begin{itemize}
    \item[(i)] $\uu_k\weaktos\uu$ in $L^\infty([0,T];H)\cap L^2([0,T];V)$,
    \item[(ii)] $\log \B_{k}\weaktos\log\B$ in $L^\infty([0,T];L^2(\Omega;\MM))$,
    \item[(iii)] $\log \B_{0,k}\weakto\log\B_0$ in $L^2(\Omega;\MM)$,
\end{itemize}
and such that the corresponding sequence of smooth solutions $\{\tC_k\}$ of \eqref{eq:Bref_evolution} with advecting field $\vc u_k$, driving $\B_k$, and initial condition $\B_{0,k}$ satisfies
\[
\log\tC_k\weaktos\log \Bref \hspace{0.4cm}\mbox{in }L^\infty([0,T];L^2(\Omega;\MM)).
\]
In particular, $\Bref$ is the limit of $\{\tC_k\}$ in $\BB_T$ with respect to the charted weak topology.
\end{definition}

\begin{thm}\label{thm:existence_charted_Bref}
Given $T>0$, for any $\B_0\in\BB$, $\vc u\in L^\infty([0,T];H)\cap L^2([0,T];V)$ and $\B\in \BB_T$, there exists a charted weak solution $\Bref\in \BB_T$ of equation \eqref{eq:Bref_evolution} with initial condition $\B_0$.
\end{thm}
\begin{rem}
The proof of Theorem~\ref{thm:existence_charted_Bref} is a mere adaptation of that of Theorem~\ref{lem:esistenza-charted} based on the availability of the estimate \eqref{est:Bref}.
\end{rem}

\begin{definition}\label{def:weak-solution2}
A triple $(\uu,\B,\Bref)$ is a {\em Leray weak solution} of the differential problem \eqref{eq:sistema-fluido} if
\begin{itemize}
    \item[(i)] $\uu\in L^\infty([0,T];H)\cap L^2([0,T];V)$ satisfies equation \eqref{eq:flow_fluid} in the distributional sense, namely
    \begin{multline}
        \dint\rho\left( \uu\,\de_t\vc\Theta+\uu\otimes\uu:\nabla\vc\Theta\right)\dee x\dee t+\int_\Omega \uu_0\vc\Theta(0,\cdot)\dee x\\
        =-\dint\left(2\eta\vt D+\kappa(\log\B-\log\Bref)\right):\nabla\vc\Theta\,\dee x \dee t,
    \end{multline}
    for all divergence-free test functions $\vc\Theta\in C^\infty_c([0,T)\times\Omega)$;
    \item[(ii)] $\B\in\BB_T$ is a charted weak solution of equation \eqref{eq:transport_B_2} with initial datum $\B_0$;
    \item[(iii)] $\Bref\in\BB_T$ is a charted weak solution of equation \eqref{eq:transport_Bref} with initial datum $\B_{0}$;
    \item[(iv)] the triple $(\uu,\B, \Bref)$ satisfies, for almost every $t\in[0,T]$, the energy inequality
    \begin{multline}
       \rho\|\uu(t,\cdot)\|_{L^2}^2+\frac{\kappa}{2}\|\log\B(t,\cdot)-\log\Bref(t,\cdot)\|_{L^2}^2+4\eta\int_0^t\|\vt D(s,\cdot)\|_{L^2}^2\,\dee s\\
       +\frac{2\kappa}{\taur}\int_0^t\|\log\B(s,\cdot)-\log\Bref(s,\cdot)\|_{L^2}^2\,\dee s
       \leq \rho\|\uu_0\|_{L^2}^2.
    \end{multline}
\end{itemize}
\end{definition}

Note that, in this case, the energy inequality features an additional term, proportional to $\kappa/\taur$, that represents the energy dissipated by the plastic evolution of $\Bref$.

\begin{thm}
Given $T>0$, for any $\uu_0\in H$ and $\B_0\in \BB$, there exists a Leray weak solution $(\uu,\B,\Bref)$ of the differential problem \eqref{eq:sistema-fluido} in the sense of Definition \eqref{def:weak-solution2}.
\end{thm}
\begin{proof}
The proof is very similar to that of Theorem \ref{thm:main1}, hence, we only describe the differences.\\
\\
\underline{\em Step 1}\hspace{0.3cm}{\bf Construction of the approximating sequence (I).}
We start by constructing and approximating sequence $\uu_n$ in the same way of {\em Step 1} of Theorem \ref{thm:main1}. We fix a $\vv\in C([0,T];H^1_0(\Omega))$ and we define $\vc v_n=\Pi_n\vv\in C([0,T]; V_n)$, where $\Pi_n:L^2(\Omega)\to V_n$ is the standard projector. Then, consider the following linear system
\begin{equation}\label{eq:sistema_approssimante_lineare_bis}
\left\{
    \begin{aligned}
    &\rho\partial_t\vc u_n+\rho\Pi_n\left[(\vc v_n\cdot\nabla)\vc u_n\right]=\eta\Pi_n\Delta\vc u_n+\kappa\Pi_n\dvg(\log\B_n-\log\Brefn),\\
    &\dvg \uu_n=0,\\
    &\partial_t\B_n+(\vc v_n\cdot\nabla)\B_n=\nabla \vc v_n\B_n+\B_n(\nabla \vc v_n)^\tsp,\\
    &\de_t\Brefn+(\vc v_n\cdot\nabla)\Brefn=\vt Q_n\Brefn+\Brefn\vt Q_n^\tsp\\ 
    &\phantom{\de_t\Brefn+}+ \frac{1}{\taur}\big[(\log\B_n-\log\Brefn)\Brefn+\Brefn(\log\B_n-\log\Brefn)\big],\\
	&\uu_n(0,\cdot)=\Pi_n \uu_0,\hspace{0.3cm}\B_n(0,\cdot)=\Brefn(0,\cdot)=\B_{0,n},
    \end{aligned}\right.
\end{equation}
where $\vt Q_n$ is a suitable projection of $\nabla\vc u_n$ described in Appendix~\ref{app:rotation}.
Since $\vc v_n$ is smooth, there exists a unique solution of the above system, at least for small times $t<T_n$. \\
\\
\underline{\em Step 2}\hspace{0.3cm}{\bf Energy estimates on $\uu_n$.}
In full analogy with {\em Step 2} of Theorem \ref{thm:main1}, we only need to prove an energy estimate on $\uu_n$. This follows form the use of Theorem \ref{lem:esistenza-charted} and of the a priori estimate \eqref{est:Bref} and leads to
\begin{multline}
    \rho\|\uu_n(t,\cdot)\|^2_{L^2}+\eta\int_0^t\|\nabla\uu_n(s,\cdot)\|^2_{L^2}\dee s\leq \rho\|\Pi_n\uu_0\|^2_{L^2}\\
    +\frac{\kappa^2}{\eta}\int_0^t\|\log\B_n(s,\cdot)-\log\Brefn(s,\cdot)\|^2_{L^2}\dee s\\
    \leq \rho\|\uu_0\|^2_{L^2}+\frac{\kappa^2}{\eta\taur^2}\left[16T(3\taur^2+32T^2)\|\nabla\vv_n\|^2_{L^2L^2}+2(2\taur^2+32T^2)\|\log\B_0\|^2_{L^2})\right],
\end{multline}
for all $0\leq t< T_n<T$. Since $\|\nabla\vv_n\|^2_{L^2L^2}$ is uniformly bounded, $\uu_n$ does not blow up in finite time.
\\
\\
\underline{\em Step 3}\hspace{0.3cm}{\bf Construction of the approximating sequence (II).}
We consider again the operator
\begin{equation}
    \mathcal{F}: \vv_n\in C([0,T];V_n)\mapsto \uu_n \in C([0,T];V_n).
\end{equation}
The proof of the equicontinuity of $\uu_n$ simply requires substituting $(\log\B_n-\log\Brefn)$ for $\log\B_n$ in the argument of Theorem \ref{thm:main1}.
Since $\log\B_n$ and $\log\Brefn$ belong to the same space, no difficulty arises.

To prove that the set $S\coloneqq\{\vv\in C([0,T];V_n):\mathcal{F}(\vv)=\lambda \vv \text{ for some }\lambda>1\}$ is bounded we proceed as follows.
We fix $\lambda>1$ and substitute $\lambda\vv_n$ in place of $\uu_n$ in system~\eqref{eq:sistema_approssimante_lineare_bis}, then we multiply the first equation by $\vv_n$ and integrate in space and time to obtain
\begin{multline}\label{bilancio-vn-bis}
    \lambda\rho\int_{\Omega}\vert\vv_n\vert^2\dee x+2\lambda\eta\int_0^t\int_{\Omega}\vert\nabla \vv_n\vert^2\dee x\dee s\\
    =\lambda\rho\int_{\Omega}\vert\Pi_n\uu_0)\vert^2\dee x-2\kappa\int_0^t\int_{\Omega}(\log\B_n-\log\Brefn):\nabla\vv_n\dee x\dee s.
\end{multline}
By taking the sum of \eqref{bilancio-vn-bis} with $\kappa/2$ times \eqref{eq:Brefid2} applied to $\B_n$ and $\Brefn$, we easily find 
\begin{equation}\label{sbilancio-vn-bis}
    \int_{\Omega}\vert\vv_n\vert^2\dee x+\frac{2\eta}{\rho}\int_0^t\int_{\Omega}\vert\nabla \vv_n\vert^2\dee x\dee s
    \leq\|\uu_0\|^2_{L^2},
\end{equation}
that provides a $\lambda$-independent bound on $\vv_n$ and the boundedness of $S$ by norm equivalence.
\\
\\
\underline{\em Step 4}\hspace{0.3cm}{\bf Convergence of the approximating sequence.}
The convergence of $\uu_n$, $\B_n$, and $\Brefn$ can be shown precisely as in {\em Step 4} of Theorem \ref{thm:main1}.
\\
\\
\underline{\em Step 5}\hspace{0.3cm}{\bf Energy inequality.}
Note that, by arguing as in the proof of {\em Step 3} with $\lambda=1$ and by recalling the identity \eqref{eq:DeqGrad}, the triple $(\uu_n,\B_n,\Brefn)$ satisfies the following identity involving kinetic, elastic, and dissipated energy:
\begin{multline}
\rho\|\uu_n(t,\cdot)\|_{L^2}^2+\frac{\kappa}{2}\|\log\B_n(t,\cdot)-\log\Brefn(t,\cdot)\|_{L^2}^2+4\eta\int_0^t\|\vt D_n(s,\cdot)\|_{L^2}^2\,\dee s\\
+\frac{2\kappa}{\taur}\int_0^t\|\log\B_n(s,\cdot)-\log\Brefn(s,\cdot)\|_{L^2}^2\,\dee s
= \rho\|\Pi_n\uu_0\|_{L^2}^2.
\end{multline}
Then, by the lower semi-continuity of the norms with respect to weak convergence and the strong convergence of the initial datum, we get that $(\uu,\B,\Bref)$ satisfies the energy inequality in Definition \ref{def:weak-solution2}.
\end{proof}

%%%%%%%%%%%%%%%%%%%%%%%%%%%%%%%
\section{Further applications}
\label{sec:further_applications}

In this section, we highlight the broader scope of applicability of charted weak solutions by recalling the form of objective rates that have been proposed in connection with non-Newtonian fluid models requiring the transport of tensorial quantities.

In this work we proposed to follow the evolution of two distinct tensor fields, $\B$ and $\Bref$, with the former linked to the kinematics of the current deformation and the latter encoding the local elastically neutral state, and then build the elastic response out of a combination of those.
A more standard approach seeks to model the evolution of the viscoelastic stress by postulating an evolution for either the stress itself or a conformation tensor, considered as a proxy for the evolving structural properties of the material. 
A cornerstone for this approach is the seminal paper by Oldroyd~\cite{Oldroyd_1950}. 

Let us denote by $\vt A$ the evolving tensor field of interest. Following Oldroyd's reasoning, in situations where dissipative interactions are dominant and microscopic inertial effects should be neglected, the material response cannot depend on a global uniform acceleration of the material points~\cite{Phan-Thien_2013}. 
Hence, particular attention should be paid to the way $\vt A$ transforms after a change of reference frame.
Given a time-dependent frame rotation $\vt{Q}$, an objective tensor $\vt{A}$ transforms as
\(
\vt{A}^* = \vt{Q} \vt{A} \vt{Q}^\tsp.
\)
In general, the rate $\adv{u}\vt{A}$ of an objective tensor is \emph{not} objective~\cite{Gurtin_2010}. Indeed, after defining the frame spin as the anti-symmetric tensor $\vt{\Omega} = (\adv{u}\vt{Q}) \vt{Q}^\tsp$ and recalling that $\vt{\vt{Q}}^\tsp \vt{\vt{Q}} = \vt{I}$, we find
\[
\adv{u}\vt{A}^* = \vt{Q} \adv{u}\vt{A} \vt{Q}^{\tsp} + (\adv{u}\vt{Q}) \vt{A} \vt{Q}^\tsp + \vt{Q} \vt{A} \adv{u}\vt{Q}^{\tsp} 
= \vt{Q} (\adv{u}\vt{A}) \vt{Q}^{\tsp} +\vt\Omega \vt A^*-\vt A^*\vt\Omega.
\]
However, by the transformation law of the anti-symmetric part of the velocity gradient, that reads
$\vt{W}^* = \vt{Q} \vt{W} \vt{Q}^\tsp + \vt{\Omega}$, it is possible to define an objective corotational rate for the tensor $\vt{A}$ as
\begin{equation} \label{eq:frameind1}
\overset{\circ}{\vt{A}} \coloneqq \adv{u}\vt{A} + \vt{A} \vt{W} - \vt{W} \vt{A}.
\end{equation}
Objectivity is equally preserved if we add objective tensors or a combination of them on the right-hand side. In particular, the symmetric part $\vt D$ of the velocity gradient $\nabla\uu$ is objective, and adding $\vt{A} \vt{D} + \vt{D} \vt{A}$ to \eqref{eq:frameind1} we can get the covariantly-convected (or lower-convected) rate 
\begin{equation}\label{eq:lower_convected}
\overset{\vt{\Delta}}{\vt{A}} \coloneqq \adv{u}\vt{A} + \vt{A} \nabla \vc{u} + \nabla \vc{u}^\tsp \vt{A},
\end{equation}
adding $\vt{A} \vt{D} - \vt{D} \vt{A}$ we obtain the contravariantly-convected rate
\begin{equation}
\overset{\diamond}{\vt{A}} \coloneqq \adv{u}\vt{A} - \nabla \vc{u} \vt{A} - \vt{A} \nabla \vc{u}^\tsp,
\end{equation}
while adding $-2\vt{A} \nabla \vc{u} -2 \nabla \vc{u}^\tsp \vt{A}$ to \eqref{eq:lower_convected} we obtain the upper-convected (or Oldroyd's) rate
\begin{equation}
\overset{\nabla}{\vt{A}} \coloneqq \adv{u}\vt{A} - \vt{A} \nabla \vc{u} - \nabla \vc{u}^\tsp \vt{A}.
\end{equation}

Evolution equations featuring any of these objective rates, widely used in non-Newtonian fluid models, can be tackled using the notion of charted weak solutions. Importantly, additional relaxation terms in the tensorial evolution should be postulated in conjunction with the form of the elastic stress in ways that lead to a good interaction with the momentum balance equation. This would allow  to obtain existence results for their coupled evolution, as is the case for the models presented above.

Finally, further directions in the construction of viscoelastic models that could benefit from our approach involve considering the role of finite extensibility of the polymeric chains that confer viscoelastic properties to several materials. This results in a strongly nonlinear response under continuous deformation and in an evolution for the relaxed state dominated by plastic effects.
The emphasis we placed on the evolution equation for the relaxed state as a basic ingredient to describe viscoelasticity represents a key step in connecting non-Newtonian fluid mechanics to plasticity theory in an Eulerian setting.

\section{Discussion of open problems}\label{sec:open_problems}

Regarding the notion of charted weak solutions introduced above, there are a few important questions that remain open. First of all, we may ask whether the weak convergence of approximate solutions in Definition~\ref{def:charted} is actually a strong convergence.
There are examples of similar equations in which it is possible to obtain such improved convergence by careful estimates of remainder terms in the sequence of norm defects~\cite{LM,Masmoudi_2011,Bejaoui_2013}.
Exploring that direction, we have found essential obstructions to the argument due to the inherent lack of integrability of the rotation term $\vt\Omega$ in the logarithmic strain equation~\eqref{eq:logB_evol}.
The potential mismatch of the eigenvectors between different terms of the weakly converging sequence $\{\log\B_k\}$ is also the main source of difficulties in proving a Cauchy property, and hence strong convergence, of that sequence. This is another instance of the deep difference between what can be achieved in the analysis of tensorial equations as opposed to scalar ones.

Related to the issue of strong convergence, there is the question of whether a charted weak solution of \eqref{eq:B} is a solution in a more classical sense. Even in the presence of a strong convergence, the lack of integrability estimates for the tensor field $\B$ would prevent interpreting \eqref{eq:B} in the sense of distributions. 
One could try to base the meaning of \eqref{eq:B} on a notion of distributional solution for the logarithmic strain equation \eqref{eq:logB_evol}, but we face again a lack of estimates on the source terms appearing in the latter equation. 
A different way to tackle this issue would be to consider equation \eqref{eq:F} for the tensor field $\F$. In the work by Kalousek~\cite{Kalousek_2019}, the author proves that if $\F_0\in L^2$ and $\dvg \F_0=0$, then one can give a distributional meaning to equation \eqref{eq:F}. However, although in principle we could assume these further properties, it is not clear how to provide a link between the charted weak solution $\B$ and the weak limit $\F$. In spite of these difficulties, we believe that it is of paramount importance to place the tensorial evolution equations in the manifold setting suggested by the mechanical meaning of the tensor fields. This can in fact provide a solid ground for the treatment of finite-deformation problems, as shown by our results that apply to arbitrarily large deformations. 

Finally, the tensorial nature of the transported quantities and the consequent lack of commutativity prevent us from proving uniqueness of charted weak solutions using direct computations, in spite of the availability of \emph{a priori} estimates. This is due to the fact that such estimates control the eigenvalues of the tensor fields (which is sufficient to obtain weak compactness) but cannot constrain the rotation of eigenvectors in an effective way. 
Obviously, one should address all of these points for charted weak solutions of \eqref{eq:B} with a given velocity field before working on possible improvements of the results regarding the coupled equations of viscoelastic models.

\appendix

%\begin{appendices}
%%%%%%%%%%%%%%%%%%%%%%%%%%%%%%%
\section{Details on \emph{a priori} estimates}
\label{app:estimates}

In this section we establish some accessory results useful in obtaining identities and estimates. We begin by considering a symmetric tensor field $\tC$ that can be diagonalized as $\tC=\vt R\vt D\vt R^\tsp$, with $\vt D$ diagonal and $\vt R$ orthogonal.
By differentiating the relation $\vt R\vt R^\tsp=\I$, we find that the tensor field $\vt\Omega_{\vt R}:=(\adv{u}\vt R)\vt R^\tsp$ is anti-symmetric. 
This entails the fundamental relation
\begin{equation}
\adv{u}\tC=\vt\Omega_{\vt R}\tC-\tC\vt\Omega_{\vt R}+\vt R(\adv{u}\vt D)\vt R^\tsp.
\end{equation}
If we now consider a matrix-valued analytic function $f$ of $\tC$ such that $f(\tC)=\vt Rf(\vt D)\vt R^\tsp$, we find that
\begin{multline}
\adv{u}f(\tC)=\vt\Omega_{\vt R}f(\tC)-f(\tC)\vt\Omega_{\vt R}+f'(\tC)\vt R(\adv{u}\vt D)\vt R^\tsp\\
=\vt\Omega_{\vt R}f(\tC)-f(\tC)\vt\Omega_{\vt R}+f'(\tC)(\adv{u}\tC-\vt\Omega_{\vt R}\tC+\tC\vt\Omega_{\vt R}).
\end{multline}

We are interested in two particular cases of matrix functions, namely the inverse $\tC^{-1}$ and the logarithm $\log\tC$. If $\det\tC\neq 0$, exploiting the commutation properties of diagonal matrices, we have
\begin{multline}\label{eq:inverse_derivative}
\adv{u}\tC^{-1}=\vt\Omega_{\vt R}\tC^{-1}-\tC^{-1}\vt\Omega_{\vt R}-\tC^{-1}(\adv{u}\tC-\vt\Omega_{\vt R}\tC+\tC\vt\Omega_{\vt R})\tC^{-1}\\
=-\tC^{-1}(\adv{u}\tC)\tC^{-1}.
\end{multline}
Moreover, if $\tC$ is symmetric and positive definite, we obtain
\begin{equation}\label{eq:log_evol}
\adv{u}\log\tC=\vt\Omega_{\vt R}\log\tC-(\log\tC)\vt\Omega_{\vt R}
-\tC^{-\frac12}\vt\Omega_{\vt R}\tC^{\frac12}+\tC^{\frac12}\vt\Omega_{\vt R}\tC^{-\frac12}
+\tC^{-\frac12}(\adv{u}\tC)\tC^{-\frac12},
\end{equation}
that leads to
\begin{align}
\frac12\adv{u}(\log\tC:\log\tC)&=(\adv{u}\log\tC):\log\tC\notag\\
&=\vt\Omega_{\vt R}\log\tC:\log\tC-(\log\tC)\vt\Omega_{\vt R}:\log\tC\notag\\
&\quad-\vt\Omega_{\vt R}:\log\tC+\vt\Omega_{\vt R}:\log\tC
+\tC^{-\frac12}(\adv{u}\tC)\tC^{-\frac12}:\log\tC\notag\\
&=\adv{u}\tC:\tC^{-1}\log\tC,\label{eq:log_derivative}
\end{align}
where we used the cyclic property of the trace, the symmetry of $\tC$, and the fact that powers of $\tC$ and $\log\tC$ commute.

%%%%%%%%%%%%%%%%%%%%%%%%%%%%%%%
\section{Rotation of principal strains}
\label{app:rotation}

We consider a given symmetric matrix $\B\in\mat_d(\bR)$ positive definite and with $\det \B=1$. Denoting by $\vc b_i$ for $i=1,\ldots,d$ orthonormal eigenvectors of $\B$, we can consider the tensors
\begin{gather*}
\vt Z_i=\vc b_i\otimes\vc b_i\qquad\text{for }i=1,\ldots,d,\\
\vt Z_{d+i}=\frac{1}{\sqrt{2}}(\vc b_i\otimes\vc b_j+\vc b_j\otimes\vc b_i)\qquad\text{for }i,j=1,\ldots,d\text{ with }j> i,\\
\vt Z_{\frac{d(d+1)}{2}+i}=\frac{1}{\sqrt{2}}(\vc b_i\otimes\vc b_j-\vc b_j\otimes\vc b_i)\qquad\text{for }i,j=1,\ldots,d\text{ with }j> i,
\end{gather*}
and form the basis $\mathcal Z=\{\vt Z_i:i=1,\ldots,d^2\}$ for $\mat_d(\bR)$, that is othonormal with respect to the scalar product $\vt A:\vt C=\tr(\vt A\vt C^\tsp)$ that defines orthogonality in $\mat_d(\bR)$.

\begin{prop}
The orthogonal complement of the subspace 
\[
\mathcal C_\B:=\{\M\in\mat_d(\bR):\M^\tsp=\M\text{ and }\M\B-\B\M=\vt 0\}
\]
is the sum of the subspaces $\mathcal A$ of anti-symmetric matrices and 
\[
\mathcal K_\B:=\{\M\in\mat_d(\bR):\M\B+\B\M^\tsp=\vt 0\}.
\]
\end{prop}

\begin{rem}
Given the defining properties of these subspaces, the projection $\vt Q$ of $\nabla\vc u$ on $\mathcal C_\B^\perp$ is solely and the only responsible for the rotation of the eigenvectors of $\B$ generated by equation \eqref{eq:Bonly}, beyond the effect of mere advection.  
Such an orthogonal projection can be effectively computed by subtracting from $\nabla\vc u$ its projections on the elements of the basis $\mathcal Z$ that generate $\mathcal C_\B$. Note that $\mathcal A$ and $\mathcal K_\B$ need not be orthogonal subspaces of $\mathcal C_\B^\perp$.
\end{rem}

\begin{proof}
To prove our claim, there are a few cases to be considered. If $\B=\vt I$, then $\mathcal C_\B$ comprises all symmetric matrices, $\mathcal K_\B$ is equal to $\mathcal A$, and $\mathcal C_\B^\perp=\mathcal A$. If we now assume that $\B$ has distinct (and positive) eigenvalues $b_i$, $i=1,\ldots,d$, and we represent matrices on the eigenbasis of $\B$ mentioned above, matrices $\M=(m_{ij})$ in $\mathcal K_\B$ are such that
\begin{itemize}
\item for $d=2$ we have $2m_{11}b_1=0$, $2m_{22}b_2=0$, and $m_{12}b_2+b_1m_{21}=0$ from which we conclude that $m_{11}=m_{22}=0$ and $m_{12}=-(b_1/b_2)m_{21}$;
\item for $d=3$ we similarly obtain $m_{11}=m_{22}=m_{33}=0$ and
\[
m_{12}=-(b_1/b_2)m_{21},\qquad m_{13}=-(b_1/b_3)m_{31},\qquad m_{23}=-(b_2/b_3)m_{32}.
\]
\end{itemize}
In both cases, since the eigenvalues are all distinct, we immediately conclude that $\mathcal A\cap\mathcal K_\B=\{\vt 0\}$.
Given that the non-null elements of  $\mathcal K_\B$ are clearly not symmetric, that subspace is not orthogonal to $\mathcal A$.
Nevertheless, the dimensions of $\mathcal K_\B$ and $\mathcal A$ are both equal to $1$, if $d=2$, or $3$, if $d=3$, and add up to the dimension of $\mathcal C_\B^\perp$.
In this case, the latter subspace is a direct sum of the former two.

This settles the matter for $d=2$, but in the three-dimensional case we need to consider what happens if two eigenvalues coincide and are distinct from the third one. Without loss of generality we can assume $b_1\neq b_2 = b_3$. We still obtain $m_{11}=m_{22}=m_{33}=0$ and
\[
m_{12}=-(b_1/b_2)m_{21},\qquad m_{13}=-(b_1/b_2)m_{31},\qquad m_{23}=-m_{32}.
\]
We easily see that the intersection between $\mathcal A$ and $\mathcal K_\B$ has dimension equal to $1$, while both subspaces have dimension $3$.
By Grassmann's formula, $\dim(\mathcal A+\mathcal K_\B)=5=\dim\mathcal C_\B^\perp$, since the degeneracy of the eigenvalues implies $\dim\mathcal C_\B=4$. This concludes our argument.
\end{proof}

Our result gives an alternate proof of a theorem given by Fattal \& Kupferman~\cite{Fattal_2004} and includes the cases in which the eigenvalues are degenerate. We stress that the linear decomposition of a matrix $\M$ as the sum of three terms, belonging respectively to $\mathcal C_\B$, $\mathcal K_\B$, and $\mathcal A$, is unique only when the three subspaces are in a direct sum. Our construction with orthogonal projections on a tensorial basis provides a general way to compute the component of $\nabla\vc u$ that generates only the rotation of the eigenvectors of $\B$.

\section{Equation for the logarithmic strain}\label{app:log-eq}

We can now use the results of the previous Appendices to deduce an evolution equation for $\log\B$ starting from the evolution equation for the left Cauchy--Green tensor $\B$.

First of all, we consider the decomposition $\nabla\uu=\vt\Omega+\vt K+\vt S$, where $\vt S$ is symmetric and commutes with $\B$, $\vt\Omega$ is anti-symmetric and generates the rotation of the eigenvectors of $\B$, and $\vt K$ is such that $\vt K\B+\B\vt K^\tsp=\vt 0$. With this, the evolution equation for $\B$ becomes
\[
\adv{u}\B=\nabla\uu\B+\B\nabla\uu^\tsp=\vt\Omega\B-\B\vt\Omega+2\vt S\B.
\]
We now substitute this result in the equation for $\log\B$ implied by \eqref{eq:log_evol} and, considering that $\vt\Omega=\vt\Omega_{\vt R}$, obtain
\begin{equation}\label{eq:logB_evol}
\adv{u}\log\B=\vt\Omega\log\B-(\log\B)\vt\Omega+2\vt S.
\end{equation}
Even though this equation may seem linear, we should keep in mind that both $\vt\Omega$ and $\vt S$ depend on $\B$ in a nontrivial way.

\subsection*{Acknowledgements} During the preparation of this manuscript G.~Ciampa has been supported by the ERC Starting Grant 101039762 HamDyWWa. G.~Ciampa is partially supported by INdAM-GNAMPA, by the projects PRIN 2020 ``Nonlinear evolution PDEs, fluid dynamics and transport equations: theoretical foundations and applications'', and PRIN2022 ``Classical equations of compressible fluids mechanics: existence and properties of non-classical solutions''. 
The work of G.~G.~Giusteri was partially supported by the National Group for Mathematical Physics (GNFM) of the Italian National Institute for Advanced Mathematics (INdAM). Moreover, G.~G.~Giusteri was partially supported by a project funded by the European Union -- NextGenerationEU under the National Recovery and Resilience Plan (NRRP), Mission 4 Component 2 Investment 1.1 - Call PRIN 2022 No.\ 104 of February 2, 2022 of Italian Ministry of University and Research; Project 202249PF73 (subject area: PE - Physical Sciences and Engineering) ``Mathematical models for viscoelastic biological matter''. 

%\end{appendices}

\bibliographystyle{abbrv}
\bibliography{references-cgs}
\end{document}